\documentclass[reqno,10pt,centertags,draft]{amsart}
\usepackage{amsmath,amsthm,amscd,amssymb,latexsym,esint,upref,enumerate,color,
verbatim,yfonts}
\date{\today}

\makeatletter
\def\theequation{\@arabic\c@equation}

\newcommand{\oT}{H}

\newcommand{\bbN}{{\mathbb{N}}}
\newcommand{\bbR}{{\mathbb{R}}}

\newcommand{\bbQ}{{\mathbb{Q}}}

\newcommand{\bbC}{{\mathbb{C}}}

\newcommand{\cB}{{\mathcal B}}

\newcommand{\cF}{{\mathcal F}}

\newcommand{\cH}{{\mathcal H}}
\newcommand{\cI}{{\mathcal I}}

\newcommand{\cK}{{\mathcal K}}

\newcommand{\cM}{{\mathcal M}}
\newcommand{\cN}{{\mathcal N}}

\newcommand{\cS}{{\mathcal S}}

\newcommand{\cV}{{\mathcal V}}

\newcommand{\no}{\nonumber}
\newcommand{\lb}{\label}
\newcommand{\f}{\frac}
\newcommand{\ul}{\underline}
\newcommand{\ol}{\overline}

\newcommand{\ti}{\tilde}
\newcommand{\wti}{\widetilde}

\newcommand{\al}{\alpha}

\newcommand{\Oh}{O}

\newcommand{\loc}{\text{\rm{loc}}}

\newcommand{\ran}{\operatorname{ran}}

\newcommand{\dom}{\operatorname{dom}}
\newcommand{\supp}{\operatorname{supp}}

\renewcommand{\Re}{\operatorname{Re}}
\renewcommand{\Im}{\operatorname{Im}}

\newcommand{\bi}{\bibitem}
\newcommand{\hatt}{\widehat}

\DeclareMathOperator*{\slim}{s-lim}

\numberwithin{equation}{section}

\newtheorem{theorem}{Theorem}[section]

\newtheorem{lemma}[theorem]{Lemma}

\newtheorem{hypothesis}[theorem]{Hypothesis}

\theoremstyle{definition}
\newtheorem{definition}[theorem]{Definition}
\newtheorem{remark}[theorem]{Remark}

\begin{document}

\title[Model Hilbert Spaces]{On a class of Model Hilbert Spaces}
\author[F.\ Gesztesy, R.\ Weikard, and M.\ Zinchenko]{Fritz
Gesztesy, Rudi Weikard, and Maxim Zinchenko}

\address{Department of Mathematics,
University of Missouri, Columbia, MO 65211, USA}
\email{gesztesyf@missouri.edu}
\urladdr{http://www.math.missouri.edu/personnel/faculty/gesztesyf.html}

\address{Department of Mathematics, University of
Alabama at Birmingham, Birmingham, AL 35294, USA}
\email{rudi@math.uab.edu}
\urladdr{http://www.math.math.uab.edu/\~{}rudi/index.html}

\address{Department of Mathematics,
University of Central Florida, Orlando, FL 32816, USA}
\email{maxim@math.ucf.edu}
\urladdr{http://www.math.ucf.edu/~maxim/}

\dedicatory{Dedicated with great pleasure to Jerry Goldstein on the occasion of his 70th birthday}
\date{\today}
\thanks{Based upon work partially supported by the US National Science
Foundation under Grant No.\ DMS-DMS 0965411.}

\begin{abstract}
We provide a detailed description of the model Hilbert space 
$L^2(\bbR; d\Sigma; \cK)$, were $\cK$ represents a complex, 
separable Hilbert space, and $\Sigma$ denotes a bounded operator-valued measure.
In particular, we show that several alternative approaches to such a construction 
in the literature are equivalent. 

These spaces are of fundamental importance in the context of perturbation theory 
of self-adjoint extensions of symmetric operators, and the spectral theory of ordinary differential operators with operator-valued coefficients. 
\end{abstract}

\maketitle


\section{Introduction} \lb{s1}

The principal purpose of this note is to recall and elaborate on the construction of 
the model Hilbert space $L^2(\bbR; d\Sigma; \cK)$ and related Banach spaces 
$L^p(\bbR; w d\Sigma; \cK)$, $p \geq 1$. Here $\cK$ represents a complex, 
separable Hilbert space, $\Sigma$ denotes a bounded operator-valued measure, and 
$w$ is an appropriate scalar nonnegative weight function. 
This model Hilbert space is known to play a fundamental role in various applications 
such as, perturbation theory of self-adjoint operators, the theory of self-adjoint 
extensions of symmetric operators, and the spectral theory of ordinary differential 
operators with operator-valued coefficients (cf.\ the end of Section \ref{s2}). 

In Section \ref{s2} we describe in detail the construction of 
$L^2(\bbR; d\Sigma; \cK)$ following the approach used in \cite{GKMT01}. We 
actually will present a slight generalization to the effect that we now explicitly permit 
that the bounded operator $T = \Sigma(\bbR)$ in $\cK$ has a nontrivial null 
space. In the last part of Section \ref{s2} we will show that our construction is 
equivalent to alternative constructions employed by 
Berezanskii \cite[Sect.\ VII.2.3]{Be68} and another 
approach originally due to Gel'fand--Kostyuchenko \cite{GK55} and 
Berezanskii \cite[Ch.\ V]{Be68}.  

It is a somewhat curious fact that in the alternative construction due to 
Berezanskii \cite[Sect.\ VII.2.3]{Be68} one has a choice in the order in  
which one takes a certain completion and a quotient with respect to a semi-inner product. 
In fact, different authors frequently chose one or the other of these two different 
routes without commenting on the equivalence of these two possibilities. Thus, we 
prove their equivalence in Appendix \ref{sA}.  

We now briefly comment on the notation used in this paper: Throughout, 
$\cH$ and $\cK$  denote separable, complex Hilbert spaces, the inner product and norm in $\cH$ are denoted by $(\cdot,\cdot)_{\cH}$ (linear in the second argument) and
$\|\cdot \|_{\cH}$, respectively. The identity operator in $\cH$ is written as 
$I_{\cH}$. We denote by $\cB(\cH)$ the Banach space of linear bounded 
operators in $\cH$. The domain, range, kernel (null space)
of a linear operator will be denoted by $\dom(\cdot)$,
$\ran(\cdot)$, $\ker(\cdot)$, 
respectively. The closure of a closable operator $S$ is denoted by $\ol S$.
The Borel $\sigma$-algebra on $\bbR$ is denoted by $\mathfrak{B}(\bbR)$.

\section{Direct Integrals and the Construction of the \\ Model Hilbert Space 
$L^2(\bbR;d\Sigma;\cK)$} \label{s2}

In this section we describe in detail the construction of the model Hilbert 
space $L^2(\bbR; d\Sigma; \cK)$ (and related Banach spaces 
$L^p(\bbR; w d\Sigma; \cK)$, $p \geq 1$, $w$ an appropriate scalar 
nonnegative weight function) following (and extending) a method first described 
in \cite{GKMT01}. 

As general background literature for the topic to follow, we refer to the theory of 
direct integrals of Hilbert spaces as presented, for instance, in \cite[Ch.\ 4]{BW83},
\cite[Ch.\ 7]{BS87}, \cite[Ch.\ II]{Di96}, \cite[Ch.\ XII]{vN51}.

Throughout this section we make the following assumptions:

\begin{hypothesis} \lb{hD.0}
Let $\mu$ denote a $\sigma$-finite Borel measure on $\bbR$,
$\mathfrak{B}(\bbR)$ the Borel $\sigma$-algebra on $\bbR$,
and suppose that $\cK$ and $\cK_{\lambda}$, $\lambda\in\bbR$, denote
separable, complex Hilbert spaces such that the dimension function 
$\bbR \ni \lambda \mapsto dim (\cK_{\lambda}) \in \bbN \cup \{\infty\}$ 
is $\mu$-measurable.
\end{hypothesis}

Assuming Hypothesis \ref{hD.0}, let
$\cS(\{\cK_\lambda\}_{\lambda\in\bbR})$
be the vector space associated with the Cartesian product 
$\prod_{\lambda\in\bbR}\cK_\lambda$ equipped with
the obvious linear structure. Elements of
$\cS(\{\cK_\lambda\}_{\lambda\in\bbR})$ are maps
\begin{equation}
f \in \cS(\{\cK_\lambda\}_{\lambda\in\bbR}), \quad 
\bbR\ni\lambda\mapsto f(\lambda)\in\cK_{\lambda},
\lb{D.1}
\end{equation}
in particular, we identify $f = \{f(\lambda)\}_{\lambda \in \bbR}$. 

\begin{definition} \lb{dD.1}
Assume Hypothesis \ref{hD.0}.
A {\it measurable family of Hilbert spaces} $\cM$
modeled on $\mu$ and
$\{\cK_\lambda\}_{\lambda\in\bbR}$ is a linear
subspace $\cM\subset\cS(\{\cK_\lambda\}
_{\lambda\in\bbR})$ such that $f\in\cM$ if and
only if the map $\bbR\ni\lambda\mapsto (f(\lambda),
g(\lambda))_{\cK_\lambda}\in\bbC$ is
$\mu$-measurable for all $g\in\cM$. \\
Moreover, $\cM$ is said to be generated by some
subset $\cF$, $\cF\subset\cM$, if for every
$g\in\cM$ we can
find a sequence of functions $h_n\in{\rm lin.span}
\{\chi_Bf\in \cS(\{\cK_\lambda\}) \,|\, B\in\mathfrak{B}(\bbR), f\in\cF\}$ with
$\lim_{n\to\infty}\|g(\lambda)-h_n(\lambda)
\|_{\cK_{\lambda}}=0$
$\mu$-a.e.
\end{definition}

The definition of $\cM$ was chosen with its maximality
in mind and we refer to Lemma \ref{lD.3} and for more details
in this respect. An explicit construction of an example of
$\cM$ will be given in Theorem \ref{tD.7}.

\begin{remark} \lb{rD.2}
The following properties are proved in a standard manner:\\
$(i)$ If $f\in\cM$, $g\in\cS(\{\cK_\lambda\}
_{\lambda\in\bbR})$ and $g=f$ $\mu$-a.e.\ then $g\in\cM$.\\
$(ii)$ If $\{f_n\}_{n\in\bbN}\in\cM$,
$g\in\cS(\{\cK_\lambda\}_{\lambda\in\bbR})$
and $f_n(\lambda)\to g(\lambda)$ as $n\to\infty$
$\mu$-a.e.\ (i.e., $\lim_{n\to\infty}\|f_n(\lambda)-
g(\lambda)\|
_{\cK_\lambda}=0$ $\mu$-a.e.) then $g\in\cM$.\\
$(iii)$ If $\phi$ is a scalar-valued $\mu$-measurable
function and $f\in\cM$ then $\phi f\in\cM$.\\
$(iv)$ If $f\in\cM$ then $\bbR\ni\lambda\mapsto
\|f(\lambda)\|_{\cK_\lambda}\in [0,\infty)$ is
$\mu$-measurable.
\end{remark}

\begin{lemma} [\cite{GKMT01}] \lb{lD.3}
Assume Hypothesis \ref{hD.0}. Suppose that 
$\{f_n\}_{n\in\bbN} \subset \cS(\{\cK_\lambda\}_{\lambda\in\bbR})$
is such that\\
$($$\alpha$$)$ $\bbR\ni\lambda\mapsto
(f_m(\lambda),f_n(\lambda))_{\cK_\lambda}\in\bbC$ is
$\mu$-measurable for
all $m,n\in\bbN$.\\
$($$\beta$$)$ For $\mu$-a.e. $\lambda\in\bbR$,
$\ol{{\rm lin.span}\{f_n(\lambda)\}}=\cK_\lambda$. \\
\noindent $($In particular, any orthonormal basis $\{e_n (\lambda)\}_{n\in\bbN}$ 
in $\cK_{\lambda}$ will satisfy $(\alpha)$ and $(\beta)$.$)$ Then setting
\begin{equation}
\cM=\{g\in\cS(\{\cK_\lambda\}_{\lambda\in\bbR})
 \, |\, (f_n(\lambda),g(\lambda))_{\cK_\lambda}
\text{ is }
\mu\text{-measurable for all }n\in\bbN\}, \lb{D.3}
\end{equation}
one has the following facts: \\
$(i)$ $\cM$ is a measurable family of Hilbert spaces.\\
$(ii)$ $\cM$ is generated by $\{f_n\}_{n\in\bbN}$. \\
$(iii)$ $\cM$ is the unique measurable family of
Hilbert spaces
containing the sequence $\{f_n\}_{n\in\bbN}$. \\
$(iv)$ If $\{g_n\}_{n\in\bbN} \subset \cM$ is any sequence satisfying
$(\beta)$ then
$\cM$ is generated by $\{g_n\}_{n\in\bbN}$.
\end{lemma}

Next, let $w$ be a $\mu$-measurable function,
$w>0$ $\mu$-a.e., and consider the space
\begin{equation}
\dot L^2(\bbR; w d\mu; \cM)= \bigg\{f\in\cM \, \bigg| \,
\int_\bbR w(\lambda)
d\mu(\lambda) \, \|f(\lambda)\|^2_{\cK_\lambda}
<\infty\bigg\}
\lb{D.6}
\end{equation}
with its obvious linear structure. On
$\dot L^2(\bbR; w d\mu; \cM)$ one defines a semi-inner
product
$(\cdot,\cdot)_{\dot L^2(\bbR; w d\mu; \cM)}$ (and hence a
semi-norm $\|\cdot\|_{\dot L^2(\bbR; w d\mu; \cM)}$) by
\begin{equation}
(f,g)_{\dot L^2(\bbR; w d\mu; \cM)}=\int_\bbR
w(\lambda) d\mu(\lambda) \,
(f(\lambda),g(\lambda))_{\cK_\lambda}, \quad
f,g\in \dot L^2(\bbR; w d\mu; \cM). \lb{D.7}
\end{equation}
That \eqref{D.7} defines a semi-inner product
immediately follows
from the corresponding properties of
$(\cdot,\cdot)_{\cK_\lambda}$ and the linearity of
the integral. Next, one defines the equivalence relation
$\sim$, for elements $f,g\in \dot L^2(\bbR; w d\mu; \cM)$ by
\begin{equation}
f\sim g \text{ if and only if } f=g \quad \text{$\mu$-a.e.}
\lb{D.7a}
\end{equation}
and hence introduces the set of equivalence classes of
$\dot L^2(\bbR; w d\mu; \cM)$ denoted by
\begin{equation}
L^2(\bbR; w d\mu; \cM)=\dot L^2(\bbR; w d\mu; \cM)/\sim .
\lb{D.7b}
\end{equation}
In particular, introducing the subspace of null functions
\begin{align}
\cN(\bbR; w d\mu; \cM)&= \big\{f\in \dot L^2(\bbR; w d\mu; \cM)
\,\big|\, \|f(\lambda)\|_{\cK_\lambda}=0 \text{ for }
\text{$\mu$-a.e. }\lambda\in\bbR\big\} \no \\
&= \big\{f\in \dot L^2(\bbR; w d\mu; \cM) \,\big|\,
\|f\|_{\dot L^2(\bbR; w d\mu; \cM)}=0\big\}, \lb{D.7c}
\end{align}
$L^2(\bbR; w d\mu; \cM)$ is precisely the quotient space
$\dot L^2(\bbR; w d\mu; \cM)/\cN(\bbR; w d\mu; \cM)$.
Denoting the equivalence class of
$f\in \dot L^2(\bbR; w d\mu; \cM)$ temporarily by $[f]$, the
semi-inner product on $L^2(\bbR; w d\mu; \cM)$
\begin{equation}
([f],[g])_{L^2(\bbR; w d\mu; \cM)}=\int_\bbR
w(\lambda)d\mu(\lambda) \,
(f(\lambda),g(\lambda))_{\cK_\lambda} \lb{D.7d}
\end{equation}
is well-defined (i.e., independent of the chosen representatives of the equivalence classes) and actually an inner product. Thus, $L^2(\bbR; w d\mu; \cM)$ is a normed space and by the usual abuse of notation we denote its elements in the
following again by $f,g$, etc.\ Moreover, $L^2(\bbR; w d\mu; \cM)$ is also complete: 

\begin{theorem} \lb{tD.4}
Assume Hypothesis \ref{hD.0}. Then the normed space 
$L^2(\bbR; w d\mu; \cM)$ is complete and hence a Hilbert space. In addition, 
$L^2(\bbR; w d\mu; \cM)$ is separable. 
\end{theorem}
That $L^2(\bbR; w d\mu; \cM)$ is complete was shown in 
\cite[Subsect.\ 4.1.2]{BW83}, \cite[Sect.\ 7.1]{BS87}, and more recently, in 
\cite{GKMT01}. Separability of $L^2(\bbR; w d\mu; \cM)$ is proved in 
\cite[Sect.\ 7.1]{BS87} (see also \cite[Subsect.\ 4.3.2]{BW83}). 

\begin{remark} \lb{rD.4}
Clearly, the analogous construction then defines the
Banach spaces $L^p(\bbR; w d\mu; \cM)$, $p\geq 1$.
\end{remark}

Thus, $L^2(\bbR; w d\mu; \cM)$ corresponds precisely to the direct 
integral of the Hilbert spaces $\cK_\lambda$ with respect to
the measure $wd\mu$ (see, e.g., \cite[Ch.\ 4]{BW83},
\cite[Ch.\ 7]{BS87}, \cite[Ch.\ II]{Di96}, \cite[Ch.\ XII]{vN51}) and is
frequently denoted by
$\int_{\bbR}^{\oplus} w(\lambda) d\mu(\lambda) \, \cK_{\lambda}$.

\smallskip

Having reviewed the construction of $L^2(\bbR; w d\mu; \cM)=
\int_{\bbR}^{\oplus} w(\lambda) d\mu(\lambda) \, \cK_{\lambda}$ in
connection with a scalar measure $w d\mu$, we now turn to the case
of operator-valued measures and recall the following definition (we refer, for  
instance, to \cite[Sects.\ 1.2, 3.1, 5.1]{BW83}, \cite[Sect.\ VII.2.3]{Be68}, 
\cite[Ch.\ 6]{BS87}, \cite[Ch.\ I]{DU77}, \cite[Ch.\ X]{DS88}, \cite{MM04} for vector-valued 
and operator-valued measures):

\begin{definition} \lb{dA.6}
Let $\cH$ be a separable, complex Hilbert space.
A map $\Sigma:\mathfrak{B}(\bbR) \to\cB(\cH)$, with $\mathfrak{B}(\bbR)$ the
Borel $\sigma$-algebra on $\bbR$, is called a {\it bounded, nonnegative,
operator-valued measure} if the following conditions $(i)$ and $(ii)$ hold: \\
$(i)$ $\Sigma (\emptyset) =0$ and $0 \leq \Sigma(B) \in \cB(\cH)$ for all
$B \in \mathfrak{B}(\bbR)$. \\
$(ii)$ $\Sigma(\cdot)$ is strongly countably additive (i.e., with respect to the
strong operator  \hspace*{5mm} topology in $\cH$), that is,
\begin{align}
& \Sigma(B) = \slim_{N\to \infty} \sum_{j=1}^N \Sigma(B_j)   \lb{A.40} \\
& \quad \text{whenever } \, B=\bigcup_{j\in\bbN} B_j, \, \text{ with } \,
B_k\cap B_{\ell} = \emptyset \, \text{ for } \, k \neq \ell, \;
B_k \in \mathfrak{B}(\bbR), \; k, \ell \in \bbN.    \no
\end{align}
Moreover, $\Sigma(\cdot)$ is called an {\it $($operator-valued\,$)$ spectral
measure} (or an {\it orthogonal operator-valued measure}) if additionally 
the following condition $(iii)$ holds: \\
$(iii)$ $\Sigma(\cdot)$ is projection-valued (i.e., $\Sigma(B)^2 = \Sigma(B)$,
$B \in \mathfrak{B}(\bbR)$) and $\Sigma(\bbR) = I_{\cH}$. 
\end{definition}

In the following, let $\Sigma:\mathfrak{B}(\bbR) \to\cB(\cK)$ be a bounded
nonnegative measure, that is, $\Sigma$ satisfies requirements $(i)$ and $(ii)$ in
Definition \ref{dA.6}. Denoting $T = \Sigma(\bbR)$, one has
\begin{equation}
0 \leq \Sigma(B) \leq T \in \cB(\cK), \quad B \in \mathfrak{B} (\bbR),    \lb{D.8}
\end{equation}
and hence
\begin{equation}
\big\|\Sigma(B)^{1/2} \xi \big\|_{\cK} \leq \big\|T^{1/2} \xi \big\|_{\cK}, \quad \xi \in \cK,    \lb{D.8A}
\end{equation}
shows that
\begin{equation}
\ker(T) = \ker\big(T^{1/2}\big) \subseteq \ker\big(\Sigma(B)^{1/2}\big) = \ker(\Sigma(B)), \quad
B \in \mathfrak{B} (\bbR).    \lb{D.8B}
\end{equation}
We will use the orthogonal decomposition
\begin{equation}
\cK = \cK_0 \oplus \cK_1, \quad \cK_0 = \ker(T), \; \,\cK_1 = \ker(T)^\bot = \ol{\ran(T)},    \lb{D.8C}
\end{equation}
and identify $f_0 = (f_0 \;\; 0)^\top \in \cK_0$ and $f_1 = (0 \; \; f_1)^\top \in \cK_1$. In particular, with
$f = (f_0 \; \; f_1)^\top$, one has $\|f\|_{\cK}^2 = \|f_0\|_{\cK_0}^2 + \|f_1\|_{\cK_1}^2$.
Then $T$ permits the $2\times 2$ block operator representation
\begin{equation}
T = \begin{pmatrix} 0 & 0 \\ 0 & T_1 \end{pmatrix}, \, \text{ with } \,
0 \leq T_1 \in \cB(\cK_1), \;\, \ker(T_1) = \{0\},     \lb{D.8D}
\end{equation}
with respect to the decomposition \eqref{D.8C}. By \eqref{D.8B} one concludes that
$\Sigma(B)$, $B \in \mathfrak{B} (\bbR)$, is necessarily of the form
\begin{equation}
\Sigma(B) = \begin{pmatrix} 0 & D^* \\ D & \Sigma_1 (B) \end{pmatrix},
\, \text{ for some } \, 0 \leq \Sigma_1 (B) \in \cB(\cK_1), \; D \in \cB(\cK_0,\cK_1),
\lb{D.8F}
\end{equation}
with respect to the decomposition \eqref{D.8C}. The computation
\begin{equation}
0 = \Sigma(B) \begin{pmatrix} f_0 \\ 0 \end{pmatrix} 
= \begin{pmatrix} 0 & D^* \\ D 
& \Sigma_1 (B)\end{pmatrix} \begin{pmatrix} f_0 \\ 0 \end{pmatrix} 
= \begin{pmatrix} 0 \\ D f_0 \end{pmatrix},
\quad f_0 \in \cK_0,    \lb{D.8G}
\end{equation}
yields $D=0$ as $f_0 \in \cK_0$ was arbitrary. Thus, $\Sigma(B)$,
$B \in \mathfrak{B} (\bbR)$, is actually also of diagonal form
\begin{equation}
\Sigma(B) = \begin{pmatrix} 0 & 0 \\ 0 & \Sigma_1 (B) \end{pmatrix},
\, \text{ for some } \, 0 \leq \Sigma_1 (B) \in \cB(\cK_1),    \lb{D.8H}
\end{equation}
with respect to the decomposition \eqref{D.8C}.

Moreover, let
$\mu$ be a control measure for $\Sigma$ (equivalently, for $\Sigma_1$), that is,
\begin{equation}
\mu(B)=0 \text{ if and only if } \Sigma(B)=0
\text{ for all } B\in\mathfrak{B}(\bbR). \lb{D.8a}
\end{equation}
(E.g., $\mu(B)=\sum_{n\in\cI}2^{-n}(e_n,
\Sigma(B)e_n)_\cK$, $B\in\mathfrak{B}(\bbR)$,
with $\{e_n\}_{n\in\cI}$ a complete orthonormal
system in $\cK$, $\cI\subseteq\bbN$, an appropriate
index set.)

The following theorem was first stated in \cite{GKMT01} under the implicit assumption that $\Sigma (\bbR) = T = I_{\cK}$. In this paper we now treat the general case $T \in \cB(\cK)$, in particular, 
we explicitly permit the existence of a nontrivial kernel of $T$:

\begin{theorem} \lb{tD.7}
Let $\cK$ be a separable, complex Hilbert space,
$\Sigma:\mathfrak{B}(\bbR) \to\cB(\cK)$ a bounded, nonnegative
operator-valued measure, and $\mu$ a control measure for $\Sigma$.
Then there are separable, complex Hilbert spaces
$\cK_\lambda$,
$\lambda\in\bbR$, a measurable family of Hilbert spaces
${\cM}_\Sigma$
modelled on $\mu$ and $\{\cK_\lambda\}
_{\lambda\in\bbR}$,
and a bounded linear map $\ul \Lambda\in\cB\big(\cK,
L^2(\bbR; d\mu; {\cM}_\Sigma)\big)$, satisfying
\begin{equation}
\|\ul \Lambda\|_{\cB(\cK,L^2(\bbR; d\mu; {\cM}_\Sigma))}
= \big\|T^{1/2}\big\|_{\cB(\cK)},      \lb{D.8b}
\end{equation}
and
\begin{equation}
\ker(\ul \Lambda) = \ker(T),    \lb{D.8c}
\end{equation}
so that the following assertions $(i)$--$(iii)$ hold: \\
$(i)$ For all $B \in \mathfrak{B} (\bbR)$, $\xi,\eta\in\cK$,
\begin{equation}
(\eta,\Sigma(B) \xi)_\cK=\int_B d\mu(\lambda) \,
((\ul \Lambda \eta)(\lambda),(\ul \Lambda \xi)(\lambda))
_{\cK_\lambda},     \lb{D.9}
\end{equation}
in particular,
\begin{equation}
(\eta, T \xi)_\cK=\int_{\bbR} d\mu(\lambda) \,
((\ul \Lambda \eta)(\lambda),(\ul \Lambda \xi)(\lambda))
_{\cK_\lambda}. \lb{D.9a}
\end{equation}
$(ii)$ Let $\cI = \{1,\dots,N\}$ for some $N\in\bbN$, or $\cI = \bbN$. 
$\ul \Lambda(\{e_n\}_{n\in\cI})$ generates ${\cM}_\Sigma$, where
$\{e_n\}_{n\in\cI}$ denotes any
sequence of linearly independent elements in $\cK$
with the property $\ol{{\rm lin.span} \{e_n\}_{n\in\cI}}=\cK$.
In particular, $\ul \Lambda(\cK)$ generates
${\cM}_\Sigma$.\\
$(iii)$  
For all $B \in \mathfrak{B} (\bbR)$ and $\xi\in\cK$,
\begin{equation}
\ul \Lambda(S(B)\xi\big)
= \{\chi_B(\lambda) (\ul \Lambda \xi)(\lambda)\}_{\lambda\in\bbR},      \lb{D.10}
\end{equation}
where $($cf.\ \eqref{D.8D} and \eqref{D.8H}$)$
\begin{equation}
S(B) = \begin{pmatrix} I_{\cK_0} & 0 \\ 0 & T_1^{-1/2} \Sigma_1 (B)^{1/2} \end{pmatrix},
\quad S(\bbR) = I_{\cK},
\end{equation}
with respect to the decomposition \eqref{D.8C}.
\end{theorem}
\begin{proof}
Since the current version of this theorem extends the earlier one in \cite{GKMT01}, 
we now repeat it for the convenience of the reader. Moreover, we will shed additional 
light on the proof of \eqref{D.10}, correcting an oversight in this connection in 
\cite{GKMT01}. Introducing
\begin{equation}
\cV={\rm lin.span}\{e_n \in \cK \,|\, n\in\cI\}, \quad  \ol \cV = \cK, 
\end{equation} 
the Radon--Nikodym theorem implies that there
exist $\mu$-measurable $\phi_{m,n}$ such that
\begin{equation}
\int_B d\mu(\lambda)\phi_{m,n}(\lambda)=
(e_m,\Sigma(B)e_n)_\cK, \quad B \in \mathfrak{B} (\bbR), 
\; m, n \in \cI. \lb{D.11}
\end{equation}
Next, suppose $v=\sum_{n=1}^N \alpha_n e_n \in\cV$,
$\alpha_n\in\bbC$, $n=1,\dots,N$, $N\in\cI$. Then
\begin{equation}
0 \leq (v,\Sigma(B)v)_\cK=\int_B d\mu(\lambda)\sum_{m,n=1}^N
\phi_{m,n}(\lambda) \ol{\alpha_m} \alpha_n, \quad B \in \mathfrak{B} (\bbR). \lb{D.12}
\end{equation}
By considering only rational linear combinations (i.e., 
$\alpha_n \in \bbQ + i \, \bbQ$)  
we can deduce the existence of a set $E \in \mathfrak{B} (\bbR)$ with 
$\mu(\bbR\backslash E) =0$ such that for $\lambda \in E$,
\begin{equation}
\sum_{m,n} \phi_{m,n}(\lambda)\ol{\alpha_m}
\alpha_n\geq 0
\text{ for all finite sequences }
\{\alpha_n\}\subset\bbC. \lb{D.13}
\end{equation}
Hence we can define a semi-inner product $(\cdot,\cdot)_{\lambda}$ on $\cV$, 
\begin{equation}
(v,w)_{\lambda}=\sum_{m,n}\phi_{m,n}(\lambda)
\ol{\alpha_m} \beta_n, \quad \lambda \in E, 
\lb{D.14}
\end{equation}
for all $v=\sum_n \alpha_n e_n, \, w=\sum_n \beta_n e_n \in \cV$. \\
Next, let $\cK_\lambda$ be the completion of $\cV/\cN_{\lambda}$ with respect
to $\|\cdot\|_{\lambda}$, where 
\begin{equation} 
\cN_{\lambda}=\{\xi\in\cV\,|\,(\xi,\xi)_{\lambda}=0\}, \quad \lambda \in E, 
\end{equation} 
and define (for convenience) $\cK_{\lambda} = \cK$ for $\lambda \in \bbR\backslash E$. 
Consider $\cS(\{\cK_\lambda\}_{\lambda\in\bbR})$,  then each $v\in\cV$ defines an 
element 
$\ul v=\{\ul v(\lambda)\}_{\lambda\in\bbR} \in\cS(\{\cK_\lambda\}_{\lambda\in\bbR})$
by
\begin{equation}
\ul v (\lambda)=v \text{ for all } \lambda\in\bbR.
\lb{D.15}
\end{equation} Again we identify an element
$v\in \cV$ with an element in
$\cV/{\cN}_{\lambda}\subseteq
\cK_{\lambda}$. Applying Lemma \ref{lD.3}, the
collection
$\{\ul e_n\}_{n\in\cI}$ then
generates a measurable family of Hilbert spaces
${\cM}_\Sigma$. If $v = \sum_{n=1}^N \alpha_n e_n \in\cV$, $\alpha_n \in \bbC$, 
$n=1,\dots,N$, $N \in \cI$, then
\begin{align}
& \|\ul v\|^2_{L^2(\bbR; d\mu; {\cM}_\Sigma)}=\int_\bbR
d\mu(\lambda) \, (\ul v(\lambda),\ul v(\lambda))_\lambda
= \int_{\bbR} d\mu(\lambda) \, \sum_{m,n=1}^N \phi_{m,n}(\lambda) 
\ol{\alpha_m} \alpha_n    \no \\
& \quad = (v, \Sigma(\bbR) v)_{\cK} = (v,Tv)_\cK= \big\|T^{1/2}v\big\|^2_\cK. \lb{D.16}
\end{align}
Hence we can define
\begin{equation}
\dot {\ul \Lambda} : \begin{cases} \cV\to L^2(\bbR; d\mu; {\cM}_\Sigma), \\
v\mapsto\dot{\ul \Lambda} v=\ul v
=\{\ul v (\lambda)=v\}_{\lambda\in\bbR},
\end{cases}       \lb{D.17}
\end{equation}
and denote by $\ul \Lambda\in\cB\big(\cK,
L^2(\bbR; d\mu; {\cM}_\Sigma)\big)$, with 
\begin{equation}
\|\ul \Lambda\|_{\cB(\cK,L^2(\bbR; d\mu; {\cM}_\Sigma))}
= \big\|T^{1/2}\big\|_{\cB(\cK)}, 
\end{equation}
the closure of $\dot {\ul \Lambda}$.
In particular, one obtains
\begin{equation}
\|{\ul \Lambda} \xi\|_{L^2(\bbR; d\mu; \cM_{\Sigma})}^2
=\int_\bbR d\mu(\lambda) \, \|({\ul \Lambda} \xi)(\lambda)\|_{\cK_\lambda}^2
= \big\|T^{1/2} \xi \big\|^2_\cK, \quad \xi \in\cK,     \lb{D.17A}
\end{equation}
and hence 
\begin{equation}
\ker(\ul \Lambda) = \ker(T) = \cK_0.
\end{equation}

Then properties $(i)$ and $(ii)$ hold and we proceed to illustrating property
$(iii)$: Introduce the operator
\begin{equation}
\wti S(B) = \begin{pmatrix} I_{\cK_0} & 0 \\ 0
& \Sigma_1 (B)^{1/2} T_1^{-1/2} \end{pmatrix}, \quad
\dom\big(\wti S(B)\big) = \cK_0 \oplus \dom\big(T_1^{-1/2}\big),
\quad B \in \mathfrak{B} (\bbR),
\end{equation}
in $\cK = \cK_0 \oplus \cK_1$. Since
\begin{align}
\cK & = \ker(T) \oplus \ol{\ran(T)}   \no \\
& = \ker(T^{1/2}) \oplus \ol{\ran(T^{1/2})} \no \\
& = \ker(T^{1/2}) \oplus \ol{\ran(T_1^{1/2})}
\end{align}
$\dom\big(T_1^{-1/2}\big) = \ran\big(T_1^{1/2}\big)$ is dense in $\cK_1$.
(Alternatively, this follows from the fact that $\ker\big(T_1^{1/2}\big) = \{0\}$
and $0 \leq T_1^{1/2}$ is self-adjoint.) Thus $\wti S(B)$ is densely defined
in $\cK$. Applying \eqref{D.8A}, the computation
\begin{align}
 \big\|\Sigma_1(B)^{1/2} T_1^{-1/2} f\big\|_{\cK_1}
& = \big\|\Sigma(B)^{1/2} \big(0 \;\; T_1^{-1/2} f)\big)^\top \big\|_{\cK}  \no \\
& \leq \big\|T^{1/2} \big(0 \;\; T_1^{-1/2} f\big)^\top \big\|_{\cK}
= \bigg\|\begin{pmatrix} 0 & 0 \\ 0 & T_1^{1/2}\end{pmatrix}
\big(0 \;\; T_1^{-1/2} f\big)^\top\bigg\|_{\cK}  \no \\
& = \|f\|_{\cK_1}, \quad f \in \dom\big(T_1^{-1/2}\big),
\end{align}
then shows that $\Sigma_1 (B)^{1/2} T_1^{-1/2}$ has a bounded extension (its closure) to all of
$\cK_1$ with
\begin{equation}
\Big\|\ol{\Sigma_1 (B)^{1/2} T_1^{-1/2}}\Big\|_{\cB(\cK_1)} \leq 1,
\quad B \in \mathfrak{B} (\bbR).
\end{equation}
Hence, also $\wti S(B)$ has a bounded extension (its closure) to all of $\cK$ with
\begin{equation}
\Big\|\ol{\wti S(B)}\Big\|_{\cB(\cK)} \leq 1, \quad B \in \mathfrak{B} (\bbR).
\end{equation}
Moreover, one computes
\begin{align}
\wti S(B) T^{1/2} = \begin{pmatrix} I_{\cK_0} & 0 \\ 0 & \Sigma_1(B)^{1/2} T_1^{-1/2} \end{pmatrix}
\begin{pmatrix} 0 & 0 \\ 0 & T_1^{1/2} \end{pmatrix}
= \begin{pmatrix} 0 & 0 \\ 0 & \Sigma_1(B)^{1/2} \end{pmatrix} = \Sigma(B)^{1/2}.
\end{align}
In particular, this also yields
\begin{equation}
\ol{\wti S(B)} T^{1/2} = \wti S(B) T^{1/2} = \Sigma(B)^{1/2},
\quad B \in \mathfrak{B} (\bbR).   \lb{D.39a}
\end{equation}
Next, we introduce
\begin{equation}
S(B) = \big(\wti S(B)\big)^* = \Big(\ol{\wti S(B)} \, \Big)^* \in \cB(\cK),
\quad B \in \mathfrak{B} (\bbR),
\end{equation}
and note that
\begin{equation}
S(\bbR) = I_{\cK}.
\end{equation}
Then one obtains
\begin{equation}
S(B) = \begin{pmatrix} I_{\cK_0} & 0 \\ 0
& \Sigma_1 (B)^{1/2} T_1^{-1/2} \end{pmatrix}^* =
\begin{pmatrix} I_{\cK_0} & 0 \\ 0
& T_1^{-1/2 }\Sigma_1 (B)^{1/2} \end{pmatrix}, \quad B \in \mathfrak{B} (\bbR),
\end{equation}
since
\begin{equation}
\big[\Sigma_1 (B)^{1/2} T_1^{-1/2} \big]^* = T_1^{-1/2 }\Sigma_1 (B)^{1/2},
\quad B \in \mathfrak{B} (\bbR).
\end{equation}
as $\Sigma_1 (B)^{1/2} \in \cB(\cK_1)$ and $T_1^{-1/2}$ is densely defined
(in fact, self-adjoint)
in $\cK_1$ (cf.\ \cite[Theorem\ 4.19\,(b)]{We80}). By the fact that
$\ol{\wti S (B)} \in \cB(\cK)$ and hence $S (B) \in \cB(\cK)$, one concludes that
$T_1^{-1/2 }\Sigma_1 (B)^{1/2} \in \cB(\cK_1)$, that is,
$\ran\big(\Sigma_1 (B)^{1/2}\big) \subseteq \ran\big(T_1^{1/2}\big)$.

Thus, taking adjoints in \eqref{D.39a} one obtains
\begin{equation}
T^{1/2} S(B) = \Sigma(B)^{1/2},  \quad B \in \mathfrak{B} (\bbR).      \lb{D.43a}
\end{equation}

Let $ \xi\in\cK$, then combinng \eqref{D.9}, \eqref{D.9a}, and \eqref{D.43a} yields
\begin{align}
(\eta, \Sigma(B) \xi)_{\cK} &= (\Sigma(B)^{1/2} \eta, \Sigma(B)^{1/2} \xi)_{\cK}
\no \\
&= (T^{1/2} S(B) \eta, T^{1/2} S(B) \xi)_{\cK}   \no \\
&= (S(B) \eta, T S(B) \xi)_{\cK}     \no \\
&= \int_{\bbR} d \mu(\lambda) \,
(({\ul \Lambda} S(B) \eta)(\lambda), ({\ul \Lambda} S(B) \xi)(\lambda))_{\cK_\lambda}
\no \\
&= \int_B d \mu(\lambda) \,
(({\ul \Lambda} \eta)(\lambda), ({\ul \Lambda} \xi)(\lambda))_{\cK_\lambda}   \no \\
&= \int_{\bbR} d \mu(\lambda) \,
(\chi_B(\lambda) ({\ul \Lambda} \eta)(\lambda), \chi_B(\lambda) ({\ul \Lambda} \xi)(\lambda))_{\cK_\lambda},   \quad B \in \mathfrak{B} (\bbR),    \lb{D.17B}
\end{align}
implying \eqref{D.10}.
\end{proof}

Implicitly in the proof of Theorem \ref{tD.7} is a special case of the following result, which appears to be of independent interest. It may well be known, but since 
we could not quickly find it in the literature we include its short proof for the convenience 
of the reader:

\begin{lemma} \lb{lD.8}
Let $\cH$ be a complex, separable Hilbert space, $F, G$ self-adjoint operators in $\cH$,
and $0 \leq F \leq G$. Then
\begin{equation}
\ran (F^{\alpha}) \subseteq \ran(G^{\alpha}), \quad
\alpha \in (0,1/2].    \lb{D.17a}
\end{equation}
In particular, if in addition $F, G \in \cB(\cH)$ and $\ker(G) =\{0\}$, then
\begin{equation}
G^{-\alpha}F^{\alpha} \in \cB(\cH), \quad \alpha \in (0,1/2].   \lb{D.17AA}
\end{equation}
\end{lemma}
\begin{proof}
The hypothesis
$0 \leq F \leq G$ implies
\begin{equation}
\|F^{1/2} x\|_{\cH} \leq \|G^{1/2} x\|_{\cH}, \quad x \in \dom(G^{1/2})
\subseteq \dom(F^{1/2}),     \lb{D.17b}
\end{equation}
and hence one concludes as before in the context of bounded operators 
(cf.\ \eqref{D.8B}) that
\begin{equation}
\ker(G) \subseteq \ker(F).
\end{equation}
Thus, in analogy to \eqref{D.8C}, \eqref{D.8D}, and \eqref{D.8H}, we again decompose
\begin{equation}
\cH = \cH_0 \oplus \cH_1, \quad \cH_0 = \ker(G), \; \,
\cH_1 = \ker(G)^\bot = \ol{\ran(G)},    \lb{D.17C}
\end{equation}
and hence obtain the $2\times 2$ block operator representations
\begin{equation}
F = \begin{pmatrix} 0 & 0 \\ 0 & F_1 \end{pmatrix},   \quad
G = \begin{pmatrix} 0 & 0 \\ 0 & G_1 \end{pmatrix}, \quad
\ker(G_1) = \{0\},   \lb{D.17D}
\end{equation}
with respect to the decomposition \eqref{D.17C}, with self-adjoint operators
$F_1, G_1$ in $\cH_1$ satisfying $0 \leq F_1 \leq G_1$. In particular,
\begin{equation}
\|F_1^{1/2} x\|_{\cH_1} \leq \|G_1^{1/2} x\|_{\cH_1}, \quad x \in \dom(G_1^{1/2})
\subseteq \dom(F_1^{1/2}),    \lb{D.17E}
\end{equation}
and Heinz's inequality (cf.\ \cite[Satz\ 3]{He51}, \cite[Theorem\ 2]{Ka52})
\eqref{D.17E} implies that
\begin{equation}
\|F_1^\al x\|_{\cH} \leq \|G_1^\al x\|_{\cH},
\quad x \in \dom(G_1^{\alpha}) \subseteq \dom(F_1^{\alpha}), \; \al\in(0,1/2].    \lb{D.17c}
\end{equation}
Then for any $y\in\dom(F_1^{\alpha})$ and
$x\in\ran(G_1^\alpha)= \dom(G_1^{-\alpha})$, one computes
using self-adjointness of $F_1^\alpha$ and \eqref{D.17c}
\begin{align}
|(F_1^\alpha y, G_1^{-\alpha} x)_{\cH_1}|
&= |(y, F_1^\al G_1^{-\al} x)_{\cH_1}| \leq \|y\|_{\cH_1} \, \|F_1^\alpha
G_1^{-\alpha} x\|_{\cH_1}   \no \\
& \leq \|y\|_{\cH_1} \, \|G_1^\alpha G_1^{-\alpha} x\|_{\cH_1}
= \|y\|_{\cH_1} \, \|x\|_{\cH_1}.
\end{align}
This implies $F_1^\alpha y \in \dom((G_1^{-\alpha})^*) = \dom(G_1^{-\alpha})
= \ran(G_1^\al) = \ran(G^\al)$ and hence \eqref{D.17a} since also
$\ran(F_1^\al) = \ran(F^\al)$.
The fact \eqref{D.17AA} then follows from the closed graph theorem.
\end{proof}

Next, we recall that the construction in Theorem \ref{tD.7} is essentially unique:

\begin{theorem} [\cite{GKMT01}] \lb{tD.9}
Suppose $\cK'_{\lambda},\ \lambda\in\bbR$ is a
family of separable
complex Hilbert spaces, $\cM'$ is a measurable
family of Hilbert spaces
modelled on $\mu$ and $\{\cK'_{\lambda}\}$, and
$\ul \Lambda'\in\cB\big(\cK,L^2(\bbR; d\mu; \cM')\big)$ is
a map satisfying $(i)$, $(ii)$, and $(iii)$ of
Theorem \ref{tD.7}.
Then for $\mu$-a.e. $\lambda\in\bbR$ there is a
unitary operator
$U_{\lambda}:\cK_{\lambda}\to\cK'_{\lambda}$
such that $f=\{f(\lambda)\}_{\lambda\in\bbR}
\in {\cM}_{\Sigma}$ if and only if
$\{U_{\lambda}f(\lambda)\}_{\lambda \in \bbR}\in\cM'$ and for all
$\xi\in\cK$,
\begin{equation}
(\ul \Lambda' \xi)(\lambda)= U_{\lambda}
(\ul \Lambda \xi)(\lambda) \quad \text{$\mu$-a.e.}
\lb{D.17g}
\end{equation}
\end{theorem}

\begin{remark} \lb{rD.11}
$(i)$ Without going into further details, we note that
${\cM}_{\Sigma}$ depends of course on the control measure $\mu$.
However, a change in $\mu$ merely effects a change in
density and so ${\cM}_{\Sigma}$ can essentially be
viewed as $\mu$-independent. \\
$(ii)$ With $0 < w$ a $\mu$-measurable weight function, one can also consider the
Hilbert space $L^2(\bbR; w d\mu; {\cM}_\Sigma)$.
In view of our comment in item $(i)$ concerning the mild
dependence on the control measure $\mu$ of
${\cM}_{\Sigma}$, one typically puts more emphasis on the
operator-valued measure $\Sigma$ and hence uses the more suggestive 
notation $L^2(\bbR; w d\Sigma; \cK)$ instead of the more
precise $L^2(\bbR; w d\mu; {\cM}_\Sigma)$ in this case.
\end{remark}

Next, let
\begin{equation} 
\cV={\rm lin.span}\{e_n \in \cK \,|\, n\in\cI\}, \quad  \ol \cV = \cK, 
\lb{D.17h}
\end{equation}
and define
\begin{equation}
{\ul \cV}_\Sigma={\rm lin.span}\big\{\chi_B \, \ul e_n\in
L^2(\bbR; d\mu; {\cM}_{\Sigma}) \, \big| \, B \in \mathfrak{B}(\bbR),
\,n\in\cI\big\}. \lb{D.18}
\end{equation}
The fact that $\{\ul e_n\}_{n\in\cI}$ generates
${\cM}_{\Sigma}$ then implies that ${\ul \cV}_\Sigma$
is dense in the Hilbert space $L^2(\bbR; d\mu; {\cM}_\Sigma)$, that is,
\begin{equation}
\ol{{\ul \cV}_{\Sigma}}=L^2(\bbR; d\mu; {\cM}_\Sigma). \lb{D.19}
\end{equation}

Since the operator-valued distribution function $\Sigma(\cdot)$ has at most
countably many discontinuities on $\bbR$, denoting by $\mathfrak{S}_{\Sigma}$ the
corresponding set of discontinuities of $\Sigma(\cdot)$, introducing the set of intervals
\begin{equation}
\cB_{\Sigma} = \{(\alpha, \beta] \subset \bbR\,|\, \alpha, \beta \in \bbR\backslash\mathfrak{S}_{\Sigma}\},
\end{equation}
the minimal $\sigma$-algebra generated by $\cB_{\Sigma}$ coincides with the Borel algebra
$\mathfrak{B}(\bbR)$. Hence one can introduce
\begin{equation}
{\wti {\ul \cV}}_\Sigma = {\rm lin.span}\big\{\chi_{(\alpha,\beta]} \, \ul e_n\in
L^2(\bbR; d\mu; {\cM}_{\Sigma}) \, \big| \,
\alpha, \beta \in \bbR\backslash\mathfrak{S}_{\Sigma}, \,n\in\cI\big\},    \lb{D.19a}
\end{equation}
which still retains the density property in \eqref{D.19}, that is,
\begin{equation}
\ol{{\wti{\ul \cV}}_{\Sigma}}=L^2(\bbR; d\mu; {\cM}_\Sigma). \lb{D.19b}
\end{equation}

In the following we briefly describe an alternative construction of $L^2(\bbR; d\Sigma; \cK)$ used by Berezanskii \cite[Sect.\ VII.2.3]{Be68} in order to identify the two constructions.

Introduce
\begin{align}
&C_{0,0}(\bbR;\cK) = \bigg\{\ul u:\bbR\to\cK\,\bigg|\, \ul u (\cdot) \text{ is strongly continuous in $\cK$}, \, \supp(u) \text{ is compact},   \no \\
& \hspace*{4.1cm}
\bigcup_{\lambda\in\bbR} \ran (\ul u(\lambda)) \subseteq \cK_{\ul u}, 
\, \dim(\cK_{\ul u}) < \infty\bigg\}
\end{align}
On $C_{0,0}(\bbR;\cK)$ one can introduce the semi-inner product
\begin{equation}
(\ul u,\ul v)_{L^2(\bbR; d\Sigma; \cK)}
= \int_{\bbR} d (\ul u(\lambda), \Sigma(\lambda) \ul v(\lambda))_{\cK},
\quad \ul u, \ul v \in C_{0,0}(\bbR;\cK),    \lb{D.35}
\end{equation}
where the integral on the right-hand side of \eqref{D.35} is well-defined in the 
Riemann--Stieltjes sense.
Introducing the kernel of this semi-inner product by
\begin{equation}
\cN = \{\ul u \in C_{0,0}(\bbR;\cK) \,|\, (\ul u, \ul u)_{L^2(\bbR; d\Sigma; \cK)} = 0\},
\lb{D.36}
\end{equation}
Berezanskii \cite[Sect.\ VII.2.3]{Be68} obtains the separable Hilbert space 
$L^2(\bbR; d\Sigma; \cK)$ as the completion of $C_{0,0}(\bbR;\cK)/\cN$ with respect 
to the inner product in \eqref{D.35} as
\begin{equation}
\hatt {L^2(\bbR; d\Sigma; \cK)} = \ol{C_{0,0}(\bbR;\cK)/\cN}.      \lb{D.37}
\end{equation}
In particular,
\begin{equation}
([\ul u], [\ul v])_{\hatt{L^2(\bbR; d\Sigma; \cK)}}
= \int_{\bbR} d (\ul u(\lambda), \Sigma(\lambda) \ul v(\lambda))_{\cK},
\quad \ul u, \ul v \in C_{0,0}(\bbR;\cK),      \lb{D.38}
\end{equation}
and (cf.\ also \cite[Corollary\ 2.6]{MM04}) \eqref{D.38} extends to piecewise continuous $\cK$-valued functions with compact support as long as the discontinuities of $\ul u$ and $\ul v$ are disjoint from the set $\mathfrak{S}_{\Sigma}$ (the set of discontinuities of $\Sigma(\cdot)$).

Since Kats' work in the case of a finite-dimensional Hilbert space $\cK$ (cf.\ \cite{Ka50}, \cite{Ka03} and also Fuhrman \cite[Sect.\ II.6]{Fu81} and Rosenberg \cite{Ro64}),
and especially in the work of Malamud and Malamud \cite{MM04}, who studied the general case
$\dim(\cK) \leq \infty$, it has become customary to interchange the order of taking the quotient with
respect to the semi-inner product and completion in this process of constructing
$\hatt{L^2(\bbR; d\Sigma; \cK)}$. More precisely, in this context one first completes $C_{0,0}(\bbR,\cK)$
with respect to the semi-inner product \eqref{D.35} to obtain a semi-Hilbert space
\begin{equation}
\wti {L^2(\bbR; d\Sigma; \cK)} = \ol{C_{0,0}(\bbR;\cK)},     \lb{D.39}
\end{equation}
and then takes the quotient with respect to the kernel of the underlying semi-inner product, as
described in method $\mathbf{(I)}$ of Appendix \ref{sA}. Berezanskii's approach in \cite[Sect.\ VII.2.3]{Be68}
corresponds to method $\mathbf{(II)}$ discussed in Appendix \ref{sA}. The equivalence of these two methods is not stated in these sources and hence we spelled this out explicitly in Lemma \ref{lE.1} in  Appendix \ref{sA}.

Next we will indicate that Berezanskii's construction of $\hatt{L^2(\bbR; d\Sigma; \cK)}$ 
(and hence the corresponding construction by Kats (if $\dim(\cK)<\infty$) and by 
Malamud and Malamud (if $\dim(\cK) \leq \infty$) is equivalent to the one in \cite{GKMT01} 
and hence to that outlined in
Theorem \ref{tD.7}:

\begin{theorem} \lb{tD.14}
The spaces $L^2(\bbR; d\Sigma; \cK)$ and $L^2(\bbR; d\mu; \cM_{\Sigma})$ are 
isometrically isomorphic.  
\end{theorem} 
\begin{proof} 
We first recall that the set ${\wti {\ul \cV}}_\Sigma$ in \eqref{D.19a}
is dense in $L^2(\bbR; d\mu; {\cM}_\Sigma)$. On the other hand, it was shown in the 
proof of Theorem\ 2.14 in \cite{MM04} that ${\wti {\ul \cV}}_\Sigma$ is also dense in
$L^2(\bbR; d\Sigma; \cK)$. The fact
\begin{align}
\|\chi_{(\alpha,\beta]} \, \ul{e_n}\|_{L^2(\bbR; d\Sigma; \cK)}^2
& = \int_{\bbR} d(\chi_{(\alpha,\beta]}(\lambda) \, e_n,
\Sigma(\lambda) \chi_{(\alpha,\beta]}(\lambda) \, e_n)_{\cK}   \no \\
& = \int_{(\alpha,\beta]} d(e_n, \Sigma(\lambda) e_n)_{\cK}   \no \\
& = (e_n, \Sigma((\alpha, \beta]) \, e_n)_{\cK}    \no \\
& = \int_{\bbR} d\mu(\lambda) \,
\|(({\ul \Lambda} (\chi_{(\alpha,\beta]}(\lambda) e_n)(\lambda)\|_{\cK_{\lambda}}^2     \no \\
& = \|\chi_{(\alpha,\beta]} \, \ul{e_n}\|_{L^2(\bbR; d\mu; \cM_{\Sigma})}^2,  \quad
\alpha, \beta \in \bbR\backslash \mathfrak{S}_{\Sigma}, \; n \in \cI,   \lb{D.40}
\end{align}
then establishes a densely defined isometry between the Hilbert spaces
$L^2(\bbR; d\Sigma; \cK)$ and $L^2(\bbR; d\mu; \cM_{\Sigma})$ which extends by 
continuity to a unitary map.  
\end{proof}

As a result, dropping the additional ``hat'' on the left-hand side of \eqref{D.37}, and hence 
just using the notation $L^2(\bbR; d\Sigma; \cK)$ for both Hilbert space constructions is consistent.

We continue this section by yet another approach originally due to Gel'fand and
Kostyuchenko \cite{GK55} and Berezanskii \cite[Ch.\ V]{Be68}. In this context
we also refer to Berezankii \cite[Sect.\ 2.2]{Be86},
Berezansky, Sheftel, and Us \cite[Ch.\ 15]{BSU96},
Birman and Entina \cite{BE67}, Gel'fand and Shilov \cite[Ch.\ IV]{GS67},
and M.\ Malamud and S.\ Malamud \cite{MM02}, \cite{MM04}: Introducing an
operator $K \in \cB_2(\cH)$ with $\ker(K) = \ker(K^*) = \{0\}$, one has the existence
of the weakly $\mu$-measurable nonnegative operator-valued function
$\Psi_K(\cdot)$ with values in $\cB_1(\cH)$, such that
\begin{align}
\begin{split}
(f, \Sigma (B) g)_{\cH} = \int_B d\mu(t) \,
\big(\Psi_K (t)^{1/2} K^{-1} f, \Psi_K (t)^{1/2} K^{-1} g\big)_{\cH},&   \\
f, g \in \dom\big(K^{-1}\big), \; B \in \mathfrak{B} (\bbR), \; \text{$B$ bounded,}&
\end{split}
\end{align}
with
\begin{equation}
\Psi_K (\cdot) = \f{d K^*\Sigma K}{d \mu}(\cdot) \, \text{ $\mu$-a.e.}
\end{equation}
In fact, the derivative $\Psi_K (\cdot)$ exists in the $\cB_1(\cH)$-norm (cf.\ \cite{BE67}
and \cite{MM02}, \cite{MM04}). Introducing the semi-Hilbert space
$\wti \cH_t$, $t\in\bbR$, as the completion of $\dom\big(K^{-1}\big)$ with respect to
the semi-inner product
\begin{equation}
(f,g)_{\wti \cH_t} = \big(\Psi_K (t)^{1/2} K^{-1} f, \Psi_K (t)^{1/2} K^{-1} g\big)_{\cH},
\quad f, g \in \dom\big(K^{-1}\big), \; t \in \bbR,
\end{equation}
factoring $\wti \cH_t$ by the kernel of the corresponding semi-norm
$\ker(\|\cdot\|_{\wti \cH_t})$ then yields the Hilbert space
$\cH_t = \wti \cH_t / \ker(\|\cdot\|_{\wti \cH_t})$, $t\in\bbR$. One can show
(cf.\ \cite{MM02}, \cite{MM04}) that
\begin{equation}
L^2(\bbR; d\Sigma; \cK) \, \text{ and } \, \int_{\bbR}^{\oplus} d \mu(t) \, \cH_t
\, \text{ are isometrically isomorphic,}
\end{equation}
yielding yet another construction of $L^2(\bbR; d\Sigma; \cK)$.

\smallskip

We conclude this section by sketching some applications to the perturbation theory of 
self-adjoint operators and to the theory of self-adjoint extensions of symmetric operators, 
following \cite{GKMT01}. We will also briefly comment on work in preparation concerning 
the spectral theory of ordinary differential operators with operator-valued coefficients.  

\medskip

\noindent 
$\mathbf{(I)}$ {\bf Self-adjoint perturbations of self-adjoint operators.} 

We start by recalling the following result:

\begin{lemma} [\cite{GKMT01}] \lb{lD.15} 
Suppose $\cK$, $\cH$ are separable complex Hilbert spaces, $K\in\cB(\cK,\cH)$, 
$\{E(B)\}_{B\in\mathfrak{B}(\bbR)}$ is a family of orthogonal projections in $\cH$, 
and assume that 
\begin{equation}
\ol{{\rm lin.span}\{E(B)Ke_n\in\cH
\,|\, B\in\mathfrak{B}(\bbR),\, n\in\cI\}}=\cH, 
\end{equation}
with $\{e_n\}_{n\in\cI}$,  $\cI\subset\bbN$ a complete orthonormal system in $\cK$.  Define
\begin{equation}
\Sigma:\Sigma\to\cB(\cK), \quad \Sigma(B)=K^*E(B)K, 
\end{equation}
and introduce
\begin{align}
&\dot U:{\ul {\cV}}_\Sigma\to\cH, \no \\
&{\ul \cV}_\Sigma\ni \sum_{m=1}^M \sum_{n=1}^N 
\alpha_{m,n}\chi_{B_m}\ul e_n 
\mapsto \dot U\bigg(\sum_{m=1}^M \sum_{n=1}^N \alpha_{m,n} 
\chi_{B_m}\ul e_n\bigg)    \\
&\hspace*{4.35cm}=\sum_{m=1}^M \sum_{n=1}^N \alpha_{m,n} 
E(B_m)Ke_n\in\cH, \no \\
&\hspace*{1.75cm} \alpha_{m,n}\in\bbC, \, m=1,\dots,M, 
\, n=1,\dots,N, \, M,N\in\cI. \no
\end{align}
Then $\dot U$ extends to a unitary operator $U:L^2(\bbR;d\mu; {\cM}_\Sigma)\to\cH$.
\end{lemma}

Next, let $H_0$ a self-adjoint (possibly unbounded) operator in $\cH$, $L$ a bounded
self-adjoint operator in $\cK$, and $K:\cK\to \cH$ a bounded operator. 

Define the self-adjoint operator $H_L$ in $\cH$,
\begin{equation}
H_L=H_0+KLK^*, \quad \dom(H_L)=\dom(H_0). 
\end{equation}
Given the perturbation $H_L$ of $H_0$,
 we introduce the associated operator-valued Herglotz function in $\cK$,
\begin{equation}
M_L(z)=K^*(H_L-z)^{-1}K, \quad z \in \bbC\backslash \bbR.  
\end{equation}

Next, let $\{ E_0(\lambda)\}_{\lambda \in \bbR}$ be the family of strongly right-continuous orthogonal spectral projections of $H_0$ in $\cH$ and suppose that $K\cK \subseteq \cH$ 
is a generating subspace for $H_0$, that is, one of the following
(equivalent) equations holds: 
\begin{align}
\cH&=\ol{{\rm lin.span}\{(H_0-z)^{-1}Ke_n\in \cH \, \vert\, n \in \cI,
\, z\in \bbC\backslash \bbR\}}    \\
&=\ol{{\rm lin.span}\{E_0(\lambda)Ke_n\in \cH \, \vert \, n \in \cI, \, \lambda\in \bbR\}}, 
\end{align}
where $\{e_n\}_{n\in\cI}$,  $\cI\subseteq\bbN$ an appropriate index set, represents 
a complete orthonormal system in $\cK$.

Denoting by $\{ E_L(\lambda)\}_{\lambda\in \bbR}$ the family of strongly right-continuous orthogonal spectral projections of $H_L$ in $\cH$ one introduces
\begin{equation}
\Omega_L(\lambda)=K^*E_L(\lambda)K, \quad \lambda\in \bbR,      \lb{2.84}
\end{equation}
and hence verifies
\begin{align}
M_L(z)&=K^*(H_L-z)^{-1}K=K^*\int_\bbR dE_L(\lambda)
(\lambda-z)^{-1}K \no \\
&=\int_\bbR d\Omega_L(\lambda)(\lambda-z)^{-1},
\quad z\in  \bbC\backslash \bbR,    \lb{2.85}
\end{align}
where the operator Stieltjes integral \eqref{2.85} converges in the norm of $\cB(\cK)$ 
(cf.\ Theorems\ I.4.2 and I.4.8 in \cite{Br71}). Since 
$\slim_{z\to i \infty}z(H_L-z)^{-1}=-I_\cH$, \eqref{2.84} implies
\begin{equation}
\Omega_L(\bbR)=K^*K. 
\end{equation}
Moreover, since
$\slim_{\lambda\downarrow -\infty}E_L(\lambda)=0$,
$\slim_{\lambda\uparrow \infty}E_L(\lambda)=I_\cH$,
one infers
\begin{equation}
\slim_{\lambda\downarrow -\infty}\Omega_L(\lambda)=0,
\quad
\slim_{\lambda\uparrow \infty}\Omega_L(\lambda)=K^*K 
\end{equation}
and $\{ \Omega_L(\lambda)\}_{\lambda\in \bbR} \subset\cB(\cK)$ is a family of 
uniformly bounded, nonnegative, nondecreasing, strongly right-continuous operators 
from $\cK$ into itself. Let $\mu_L$ be a $\sigma-$finite control measure on $\bbR$ 
defined, for instance, by
\begin{equation}
\mu_L(\lambda)=\sum_{n\in\cI} 2^{-n}
(e_n,\Omega_L(\lambda)e_n)_\cK, \quad \lambda\in\bbR, 
\end{equation}
where $\{e_n\}_{n\in\cI}$ denotes a complete 
orthonormal 
system in $\cK$,  and then introduce 
$L^2(\bbR; d\mu_L; {\cM}_{\Omega_L}) \equiv L^2(\bbR, d \Omega_L; \cK)$ as in 
Remark \ref{rD.11}\,$(ii)$, replacing the pair $(\Sigma,\mu)$ by $(\Omega_L,\mu_L)$, 
etc.\ Abbreviating $\hatt\cH_L=L^2(\bbR; d \Omega_L; \cK)$, we introduce the unitary 
operator $U_L:\hatt \cH_L\to\cH$, as the operator $U$ in Lemma \ref{lD.15} and define
$\hatt H_L$ in $\hatt{\cH}_L$ by
\begin{equation}
(\hatt H_L \hat f)(\lambda)=\lambda\hat f(\lambda),
\quad
\hat f \in \dom (\hatt H_L)=L^2(\bbR; (1+\lambda^2)d\Omega_L; \cK). 
\end{equation}

The following result yields a spectral representation (diagonalization) of $H_L$:

\begin{theorem} [\cite{GKMT01}] \lb{t2.16} 
The operator $H_L$ in $\cH$ is unitarily equivalent to $\hatt H_L$ in $\hatt{\cH}_L$, 
\begin{equation}
H_L=U_L\hatt H_LU_L^{-1}. 
\end{equation}
The family of strongly right-continuous orthogonal spectral projections 
$\{\hatt E_L(\lambda)\}_{\lambda\in \bbR}$ of $\hatt H_L$ in $\hatt{\cH}_L$ is given by
\begin{equation}
(\hatt E_L(\lambda)\hat f)(\nu)=
\theta(\lambda-\nu)\hat f(\nu)
\text{ for } \Omega_L-\text{a.e. } \nu\in \bbR,
\quad \hat f \in\hatt{\cH}_L, 
\end{equation}
where $\theta(x)=\begin{cases} 1,&x\geq 0, \\
  0,&x < 0. \end{cases}$ 
\end{theorem}

For a variety of additional results in this context we refer to \cite{GKMT01}.

\medskip

\noindent 
$\mathbf{(II)}$ {\bf Self-adjoint extensions of symmetric operators.} 

We start by developing the following analog of Lemma \ref{lD.15}: Suppose $\cN$ is a separable complex Hilbert space and $\wti \Sigma:\mathfrak{B}(\bbR) \to \cB(\cN)$ 
a positive measure. Assume 
\begin{equation}
\wti \Sigma (\bbR)=\wti T\geq 0, \quad \wti T\in \cB(\cN),    
\end{equation} 
and let $\ti \mu$ be a control measure for $\wti \Sigma$. Moreover, let $\{u_n\}_{n\in\cI}$,  
$\cI\subseteq\bbN$ be a sequence of linearly independent elements in $\cN$ with the 
property $\ol{{\rm lin.span}\{u_n\}_{n\in\cI}}=\cN$.  As discussed in Theorem \ref{tD.7}, 
this yields a measurable family of Hilbert spaces $\cM_{\wti \Sigma}$ modelled on 
$\ti \mu$ and $\{\cN_\lambda\}_{\lambda\in\bbR}$ and a bounded map 
$\ul \Lambda\in\cB(\cN, L^2(\bbR; d\ti \mu; \cM_{\wti \Sigma})$,  
$\|\ul \Lambda\|_{\cB(\cN, L^2(\bbR; d\ti \mu; \cM_{\wti \Sigma})}=
\|{\wti T}^{1/2}\|_{\cB(\cN)}$, such that 
$\ul \Lambda (\{u_n\}_{n\in\cI})$ generates $\cM_{\wti \Sigma}$ 
and 
\begin{equation}
\ul \Lambda:\cV\to L^2(\bbR; d\ti \mu; \cM_{\wti \Sigma}), \quad 
v\mapsto\ul \Lambda v = \ul v =\{\ul v (\lambda)=v\}_{\lambda\in\bbR},    
\end{equation}
where
\begin{equation}
\cV = {\rm lin.span}\{u_n\}_{n\in\cI}. 
\end{equation}
Each $v\in\cV$ defines an element 
\begin{equation}
\ul {\ul v}=\{\ul {\ul v}(\lambda)=(\lambda-i)^{-1}v\}_{\lambda\in\bbR} 
\in \cS (\{\cN_\lambda\}_{\lambda\in\bbR}),  
\end{equation}
and introducing the weight function
\begin{equation}
w_1(\lambda)=1+\lambda^2, \quad \lambda\in\bbR, 
\end{equation}
and Hilbert space $L^2(\bbR; w_1 d\ti \mu; \cM_{\wti \Sigma})$, one computes 
\begin{equation}
\|\ul {\ul v}\|^2_{L^2(\bbR; d\ti \mu; \cM_{\wti \Sigma})}=\int_\bbR d\ti \mu (\lambda) 
\|\ul v(\lambda)\|^2_{\cN_\lambda} = \big(v,\wti T v\big)_\cN 
= \big\|\big({\wti T}\big)^{1/2} v\big\|^2_\cN. 
\end{equation}
Thus, the linear map
\begin{equation}
\ul {\ul {\dot \Lambda}}:\cV\to L^2(\bbR; w_1 d\ti \mu; \cM_{\wti \Sigma}), \quad 
v\mapsto\ul {\ul {\dot \Lambda}}v=\ul {\ul v}=
\big\{\ul {\ul v}(\lambda)=(\lambda-i)^{-1} v\big\}
_{\lambda\in\bbR} 
\end{equation}
extends to $\ul {\ul \Lambda}\in \cB(\cN, L^2(\bbR; w_1 d\ti \mu; \cM_{\wti \Sigma})$,  
$\|\ul {\ul \Lambda}\|_{\cB(\cN, L^2(\bbR; w_1 d\ti \mu; \cM_{\wti \Sigma})}
= \big\|\big({\wti T}\big)^{1/2}\big\|_{\cB(\cN)}$. Introducing 
\begin{equation}
{\ul {\ul \cV}}_{\wti \Sigma}={\rm lin.span} \big\{\chi_B \ul {\ul {u_n}} \in 
L^2(\bbR; w_1d\ti \mu; \cM_{\wti \Sigma}) \,|\,B\in \mathfrak{B}(\bbR),\, n\in\cI\big\},  
\end{equation}
one infers that ${\ul {\ul \cV}}_{\wti \Sigma}$ is 
dense in $L^2(\bbR; w_1 d\ti \mu; \cM_{\wti \Sigma})$,  
that is,
\begin{equation}
\ol{{\ul {\ul V}}_{\wti \Sigma}}=
L^2(\bbR; w_1 d\ti \mu; \cM_{\wti \Sigma}). 
\end{equation}

\begin{lemma} [\cite{GKMT01}] \lb{l2.17}  
Suppose $\cH$ is a separable complex Hilbert space, $\cN$ a closed linear subspace 
of $\cH$, $P_\cN$ the orthogonal projection in $\cH$ onto $\cN$, 
$\{E(B)\}$,  ${B\in \mathfrak{B}(\bbR)}$ a family of orthogonal projections in $\cH$,  and 
assume
\begin{equation}
\ol{{\rm lin.span}\{E(B)u_n\in\cH\,|\,B\in \mathfrak{B}(\bbR), \,n\in\cI \}}=\cH, 
\end{equation}
with $\{u_n\}_{n\in\cI}$,  $\cI\subseteq\bbN$ a complete orthonormal system in $\cN$. 
Define 
\begin{equation}
\wti \Sigma: \mathfrak{B}(\bbR) \to\cB(\cN), \quad 
\wti \Sigma (B)=P_\cN E(B) P_\cN\big|_\cN, 
\end{equation}
and introduce
\begin{align}
&\dot {\wti U}:{\ul {\ul \cV}}_\Sigma \to \cH, \no \\
&{\ul {\ul \cV}}_\Sigma\ni \sum_{m=1}^M \sum_{n=1}^N 
\alpha_{m,n}\chi_{B_m}\ul {\ul u}_n 
\mapsto \dot {\wti U}\bigg(\sum_{m=1}^M \sum_{n=1}^N 
\alpha_{m,n} 
\chi_{B_m}\ul {\ul u}_n\bigg)  \\
&\hspace*{4.4cm}=\sum_{m=1}^M \sum_{n=1}^N \alpha_{m,n} 
E(B_m)u_n\in\cH, \no \\
&\hspace*{1.55cm} \alpha_{m,n}\in\bbC, \, m=1,\dots,M, 
\, n=1,\dots,N, \, M,N\in\cI. \no
\end{align}
Then $\dot {\wti U}$ extends to a unitary operator  
$\wti U:L^2(\bbR; w_1 d\ti \mu; \cM_{\wti \Sigma})\to\cH$.
\end{lemma}

Next, let $\dot H:\dom(\dot H)\to \cH$, $\ol{\dom (\dot H)}=\cH$ be a densely 
defined closed symmetric linear operator with equal deficiency indices
$\text{def} (\dot H)=(k,k), k\in \bbN \cup \{\infty \}$.  The deficiency subspaces
$\cN_{\pm}$ of $\dot H$ are given by
\begin{equation}
\cN_{\pm}=\ker({\dot H}^*\mp i), \quad \dim_\bbC (\cN_{\pm})=k. 
\end{equation}
In addition, let $H$ be any self-adjoint extension $H$ of $\dot H$ in $\cH$, 
$\cN$ a closed linear subspace of $\cN_+$, $\cN\subseteq \cN_+$, and 
introduce the Weyl--Titchmarsh operator $M_{H,\cN}(\cdot) \in\cB(\cN)$ associated 
with the pair $(H,\cN)$ by
\begin{align}
M_{H,\cN}(z)&=P_\cN (zH+I_\cH)(H-z)^{-1} P_\cN\big\vert_\cN \no \\
&=zI_\cN+(1+z^2)P_\cN(H-z)^{-1} P_\cN\big\vert_\cN\,, \quad  z\in \bbC\backslash \bbR, 
\end{align}
with $I_\cN$ the identity operator in $\cN$ and $P_\cN$ the orthogonal projection 
in $\cH$ onto $\cN$. The Herglotz property of $M_{H,\cN}(\cdot)$ (i.e., 
$\Im(M_{H,\cN}(z)) \geq 0$ for all $z\in\bbC_+=\{z\in\bbC\,|\, \Im(z)>0\}$) then 
yields (cf., e.g., \cite[Appendix\ A]{GWZ11}) 
\begin{equation}
M_{H,\cN}(z)= \int_\bbR d\Omega_{H,\cN}(\lambda)\bigg[\f{1}{\lambda-z} -
\f{\lambda}{1+\lambda^2}\bigg], \quad z\in\bbC\backslash\bbR,   \lb{2.106}
\end{equation}
where
\begin{align}
&\Omega_{H,\cN}(\lambda)=(1+\lambda^2) 
\big(P_\cN E_H(\lambda)P_\cN\big\vert_\cN\big),    \\
&\int_\bbR d\Omega_{H,\cN}(\lambda) (1+\lambda^2)^{-1}=I_\cN,    \\
&\int_\bbR d(\xi,\Omega_{H,\cN} (\lambda)\xi)_\cN=\infty \text{ for all }
 \xi\in \cN\backslash\{0\},   
\end{align}
and the Herglotz--Nevanlinna representation \eqref{2.106} is valid in the strong
operator topology of $\cN$. 

Next we will prepare some material that eventually
will lead to a model for the pair
$(\dot H, H)$. Let $\{u_n\}_{n\in\cI}$, $\cI\subseteq\bbN$ be a complete 
orthonormal system in $\cN$, $\{\wti \Omega (\lambda)\}_{\lambda\in \bbR}$ a
family of strongly right-continuous nondecreasing $\cB(\cN)$-valued functions
normalized by
\begin{equation}
\wti\Omega (\bbR)=I_\cN, 
\end{equation}
with the property
\begin{equation}
\int_\bbR d(\xi,\wti\Omega (\lambda)\xi)_\cN (1+\lambda^2)=\infty
 \text{ for all }  \xi\in \cN\backslash\{0\}.  
\end{equation}
Introducing the control measure $\ti\mu(B) = 
\sum_{n\in\cI} 2^{-n}(u_n,\wti\Omega(B)u_n)_\cN$, $B\in \mathfrak{B}(\bbR)$, and 
$\ul \Lambda$ as in Theorem \ref{tD.7}, we may define 
$L^p(\bbR; wd\wti\Omega; \cN)$,  $p\ge 1$,  $w\geq0$ an appropriate weight 
function. Of special importance in this section are weight functions 
of the type $w_r (\lambda)=(1+\lambda^2)^r$, $r\in\bbR$,  $\lambda\in\bbR$. 
In particular, introducing 
\begin{equation}
\Omega(B)=\int_B (1+\lambda^2)d\ti\mu(\lambda)
\f{d\wti\Omega}{d\ti\mu}(\lambda), \quad B\in \mathfrak{B}(\bbR), 
\end{equation}
we abbreviate $\hatt{\cH}=L^2(\bbR; d\Omega; \cN)$ and define the self-adjoint operator 
$\hatt H$ in $\hatt{\cH}$,
\begin{equation}
(\hatt H\hat f)(\lambda)=\lambda\hat f(\lambda), \quad \hat f \in \dom (\hatt H)=
L^2(\bbR; (1+\lambda^2) d\Omega; \cN), \lb{2.113}
\end{equation}
with corresponding family of strongly
right-continuous orthogonal
 spectral projections
\begin{equation}
(E_{\hat H}(\lambda)\hat f)(\nu) = \theta(\lambda-\nu)\hat f (\nu)
\text{ for } \Omega-\text{a.e. } \nu \in \bbR, \quad \hat f \in \hatt{\cH}. \lb{2.114}
\end{equation}
Associated with $\hatt H$ we consider the linear operator $\hatt{\dot H}$ in
$\hatt{\cH}$ defined as the following restriction of $\hatt H$
\begin{align}
&\dom (\hatt{\dot H})= \bigg\{\hat f\in \dom (\hatt H) \,\bigg\vert \,  \int_\bbR 
(1+\lambda^2)d\ti\mu(\lambda) (\ul \xi, \hat f (\lambda))_{\cN_\lambda}=0 
 \text{ for all }\ul \xi\in\ul \Lambda(\cN)\bigg\}, \no \\
&\hatt{\dot H}=\hatt H
\big\vert_{\dom(\hat{\dot H})}. \lb{2.115}
\end{align}
Here we used the notation introduced 
in the proof of Theorem \ref{tD.7}, 
\begin{equation}
\ul \xi=\ul \Lambda \xi =\{\ul \xi(\lambda) = \xi\}_{\lambda\in\bbR}.   
\end{equation}

Moreover, introducing the scale of Hilbert spaces
${\hatt{\cH}}_{2r}=L^2(\bbR; (1+\lambda^2)^r d\Omega); \cN$, $r\in \bbR$,
$\hatt{\cH}_0=\hatt{\cH}$, we consider the unitary operator $R$ from $\hatt{\cH}_2$ to
$\hatt{\cH}_{-2}$, 
\begin{align}
&R:\hatt{\cH}_2 \to \hatt{\cH}_{-2},\quad \hat f
\mapsto(1+\lambda^2)\hat f,      \\
&(\hat f, \hat g)_{\hat{\cH}_2} = (\hat f, R \hat g)_{\hat{\cH}}=
(R\hat f, \hat g)_{\hat{\cH}} = (R\hat f, R\hat g)_{{\hat{\cH}}_{-2}},
\quad \hat f, \hat g \in \hatt{\cH} _2,      \\
&(\hat u, \ul v)_{{\hat{\cH}}_{-2}} = (\hat u, R^{-1} \ul v)_{{\hat{\cH}}_{2}}=
(R^{-1}\hat u, \ul v)_{\hat{\cH}} = (R^{-1}\hat u, R^{-1}\ul v)_{\hat{\cH}_2},
\quad \hat u, \ul v \in {\hatt{\cH}}_{-2}.     
\end{align}
In particular,
\begin{equation}
\ul {\ul \Lambda}(\cN)\subset\hatt\cH, \quad
\ul \Lambda(\cN)\subset \hatt{\cH}_{-2}, 
\quad \ul \xi \in\ul \Lambda(\cN)\backslash\{0\}
\Rightarrow \ul \xi\not\in \hatt{\cH}.    
\end{equation} 

\begin{theorem}  [\cite{GKMT01}] \lb{t2.18} 
The operator $\hatt{\dot H}$ in \eqref{2.115} is densely defined symmetric and 
closed in $\hatt{\cH}$. Its deficiency indices are given by
\begin{equation}
\rm{def} (\hatt{\dot H})=(k,k), \quad k = \dim_\bbC(\cN)\in \bbN\cup\{\infty\},  
\end{equation}
and
\begin{equation}
\ker (\hatt{\dot  H}^*-z)=\ol{{\rm lin.span} \big\{\{(\lambda-z)^{-1}e_n\}_{\lambda\in\bbR}
\in \hatt{\cH} \, \big\vert \, n\in \cI \big\}}, \quad z\in\bbC\backslash\bbR. 
\end{equation}
\end{theorem}

Next, we $\cH$ decomposes into the direct orthogonal sum
\begin{equation}
\cH=\cH_0\oplus \cH_0^\bot, \quad \ker ({\dot H}^*-i)\subset \cH_0, \quad
 z\in \bbC\backslash\bbR,
\lb{2.123}
\end{equation}
where $\cH_0$ and $ \cH_0^\bot$ are invariant subspaces for all self-adjoint 
extensions of $\dot H$, that is,
\begin{equation}
(H-z)^{-1}\cH_0\subseteq \cH_0 , \quad
(H-z)^{-1}\cH_0^\bot\subseteq \cH_0^\bot, \quad
 z\in \bbC\backslash\bbR,   
\end{equation}
for all self-adjoint extensions $H$ of $\dot H$ in $\cH$. In the following we call a 
densely defined closed symmetric operator $\dot H$ with deficiency indices 
$(k,k)$, $k\in \bbN\cup\{\infty\}$ {\it prime} if
$\cH_0^\bot=\{0\}$ in the decomposition \eqref{2.123}.

Given these preliminaries we can now recall the model for the pair $(\dot H, H)$:
 
\begin{theorem}  [\cite{GKMT01}] \lb{t2.19} 
Assume $H$ is a self-adjoint extension of $\dot H$ in $\cH$ and let 
$\{ E_H(\lambda) \}_{\lambda\in \bbR}$ be the associated family of strongly 
right-continuous orthogonal spectral projections of $H$ and define
the unitary operator
$\wti U:\hatt{\cH}=L^2(\bbR; d \Omega_{H,\cN_+}; \cN_+)\to \cH$ as the operator 
$\wti U$ in Lemma \ref{l2.17}, where
\begin{equation}
\Omega_{H,\cN_+}(\lambda)=(1+\lambda^2) \big(P_{\cN_+}E_H(\lambda)P_{\cN_+}
\big\vert_{\cN_+}\big),    
\end{equation}
with $P_{\cN_+}$ the orthogonal projection onto $\cN_+ = \ker ( {\dot H}^*-i)$.
Then the pair $(\dot H, H)$ is unitarily equivalent to the pair 
$\big(\hatt{\dot H}, \hatt H\big)$,
\begin{equation}
\dot H=\wti U \hatt{\dot H}{\wti U}^{-1}, \quad H=\wti U\hatt H {\wti U}^{-1},    
\end{equation}
where $\hatt H$ and $\hatt{\dot H}$ are defined in \eqref{2.113}--\eqref{2.115}, and 
$\cN$ is identified with $\cN_+$, etc. Moreover,
\begin{equation}
\wti U\hatt{\cN}_+=\cN_+,     
\end{equation}
where
\begin{equation}
{\hatt \cN}_+=\ol{{\rm lin.span} \big\{\ul{\ul u}_{+,n}
\in \hatt{\cH} \,\big\vert \, \ul{\ul u}_{+,n}(\lambda)=
(\lambda-i)^{-1}u_{+,n}, \, \lambda\in\bbR, \, n\in\cI \big\}},  
\end{equation}
with $\{u_{+,n}\}_{n\in\cI}$ a complete orthonormal system in ${\cN}_+=\ker({\dot H}^*-i)$. 
\end{theorem}

Again, for a variety of additional results in this context we refer to \cite{GKMT01}. We also 
note that additional applications of operator-valued Herglotz functions (such as \eqref{2.85}
and \eqref{2.106}) can be found, for instance, in \cite{AL95}--\cite{ABT11}, 
\cite{BT92}, \cite{BMN02}, 
\cite{Ca76}, \cite{DM91}--\cite{DM97}, \cite{GKMT01}--\cite{GMZ07}, 
\cite{GT00}, \cite{GWZ11}, \cite{Go68}, \cite{KO77}, \cite{KO78}, \cite{MN11}, 
\cite{MN11a}, \cite{Mo09}--\cite{Na94}, \cite{Na74}--\cite{NA75}, \cite{Sh71},  
\cite{Tr00}, 

\medskip

\noindent 
$\mathbf{(III)}$ {\bf Spectral theory for Schr\"odinger operators with operator-valued potentials.} 

Assuming $(a,b)\subseteq\bbR$, and supposing that 
\begin{align} 
& \text{$V:(a,b)\to\cB(\cH)$ is a weakly measurable operator-valued function,}  \no \\
& \|V(\cdot)\|_{\cB(\cH)}\in L^1_\loc((a,b);dx),   \lb{2.129} \\
& V(x) = V(x)^* \,  \text{ for a.e.\ $x \in (a,b)$,}   \no 
\end{align} 
we recently developed Weyl--Titchmarsh 
theory for certain self-adjoint operators $H_{\alpha}$ in $L^2((a,b);dx;\cH)$ associated 
with the operator-valued differential expression $\tau =-(d^2/dx^2)+V(\cdot)$ 
(cf.\ \cite{GWZ11}). These are suitable restrictions of the {\it maximal} operator 
$\oT_{\max}$ in $L^2((a,b);dx;\cH)$ defined by
\begin{align}
& \oT_{\max} f = \tau f,   \no \\
& f\in \dom(\oT_{\max})=\big\{g\in L^2((a,b);dx;\cH) \,\big|\, g\in 
W^{2,1}_{\rm loc}((a,b);dx;\cH); \\ 
& \hspace*{6.6cm} \tau g\in L^2((a,b);dx;\cH)\big\}.    \no 
\end{align}
In particular, assuming in addition that $a$ is a regular endpoint for $\tau$ and $b$ 
is of limit-point type for $\tau$, the operator $H_\alpha$ defined as the restriction of
$\oT_{\max}$ by 
\begin{equation}
\dom(\oT_{\alpha})=\{u\in\dom(\oT_{\max}) \,|\, \sin(\alpha)u'(a) + \cos(\alpha)u(a)=0\}.  
\lb{3.29A}
\end{equation}
where $\alpha = \alpha^* \in\cB(\cH)$ is self-adjoint in $L^2((a,b);dx;\cH)$. Conversely, 
all self-adjoint restrictions of $\oT_{\max}$ arise in this manner. 

Introducing $\theta_\alpha(z,\cdot, x_0,), \phi_\alpha(z,\cdot,x_0)$ as those 
$\cB(\cH)$-valued solutions of $\tau Y=z Y$ which satisfy the initial conditions
\begin{equation}
\theta_\alpha(z,x_0,x_0)=\phi'_\alpha(z,x_0,x_0)=\cos(\alpha), \quad
-\phi_\alpha(z,x_0,x_0)=\theta'_\alpha(z,x_0,x_0)=\sin(\alpha),     \lb{2.131}
\end{equation}
one of the principal results in \cite{GWZ11} establishes the existence of 
$\cB(\cH)$-valued Weyl--Titchmarsh solutions $\psi_{\alpha}(z,\cdot)$ of 
$\tau Y=z Y$ of the form
\begin{equation} \label{3.58A}
\psi_{\alpha}(z,x)=\theta_{\alpha}(z,x,a)+\phi_{\alpha}(z,x,a)m_{\alpha}(z),
\quad z \in \bbC\backslash\bbR, \; x \in [a,b), 
\end{equation}
satisfying 
\begin{equation}
\int_a^b dx \, \|\psi_{\alpha}(z,x) f\|_{\cH}^2 < \infty, \quad
f \in \cH, \; z \in \bbC\backslash\bbR.
\end{equation}
Moreover, the $\cB(\cH)$-valued-valued Weyl--Titchmarsh function $m_{\alpha}(\cdot)$ is 
shown to be a Herglotz function in \cite{GWZ11}. Consequently, $m_{\alpha}(\cdot)$ 
permits the Herglotz--Nevanlinna representation (cf.\ \cite[Appendix\ A]{GWZ11})
\begin{align}
& m_{\alpha}(z) = C_{\alpha} + D_{\alpha} z 
+ \int_{\bbR} d\Omega_{\alpha} (\lambda ) \bigg[\f{1}{\lambda-z}
- \f{\lambda}{1+\lambda ^2}\bigg],
\quad z\in\bbC_+,    \lb{A.42} \\
& \Omega_{\alpha} ((-\infty, \lambda]) = \slim_{\varepsilon \downarrow 0}
\int_{-\infty}^{\lambda + \varepsilon} \f{d \Omega_{\alpha} (t)}{1 + t^2},  \quad 
\lambda \in \bbR,  \\
& \Omega_{\alpha} (\bbR) = \Im(m_{\alpha} (i)) 
= \int_{\bbR} \f{d\Omega_{\alpha} (\lambda)}{1+\lambda^2} \in \cB(\cH),   \lb{A.42a} \\
& C_{\alpha} = \Re(m_{\alpha} (i)),\quad D_{\alpha} = \slim_{\eta\uparrow \infty} \,
\frac{1}{i\eta}m_{\alpha} (i\eta) \geq 0,      \lb{A.42b}
\end{align}
valid in the strong sense in $\cH$. The function $m_{\alpha}(\cdot)$ contains all the 
spectral information of $H_{\alpha}$ and is closely related to the Green's 
function of $H_{\alpha}$ as discussed in \cite{GWZ11}.   

Introducing the Hilbert space $L^2(\bbR; d\Omega_{\alpha}; \cH)$, the spectral 
representation (resp., model representation) of $H_{\alpha}$ then aims at exhibiting 
the unitary equivalence of $H_{\alpha}$ with the operator of multiplication by the 
independent variable in $L^2(\bbR; d\Omega_{\alpha}; \cH)$. Under stronger 
hypotheses on $V$ than those recorded in \eqref{2.129}, for instance, continuity 
of $V(\cdot)$ in $\cB(\cH)$, such a result has been shown by Rofe-Beketov \cite{Ro60} 
and Gorba{\v c}uk \cite{Go68}, and subsequently, under 
hypotheses close to those in \eqref{2.129}, by Saito \cite{Sa71}. Our own approach to 
this circle of ideas is in preparation \cite{GWZ12}. 

We conclude this section by noting that certain classes of unbounded operator-valued 
potentials $V$ lead to applications to multi-dimensional Schr\"odinger operators in 
$L^2(\bbR^n; d^n x)$, $n \in \bbN$, $n \geq 2$. It is precisely the connection between 
multi-dimensional Schr\"odinger operators and one-dimensional Schr\"odinger operators 
with unbounded operator-valued potentials which originally motivated our interest in this 
area.

\appendix
\section{Completion of Semi-Metric Spaces} \lb{sA}
\setcounter{theorem}{0}
\setcounter{equation}{0}

In this appendix we establish the equivalence of two approaches to the procedure of completion of semi-metric spaces. For background material on semi-metric spaces we 
refer to \cite[Ch.\ 9]{Wi83}.

We start by recalling that a semi-metric space $(S,\rho)$ is a set $S$ with a semi-metric
$\rho: S \times S \to [0,\infty)$, satisfying
\begin{align}
& \rho(x,x)=0, \quad x \in S,    \lb{E.0a} \\
& \rho(x,y)=\rho(y,x), \quad x,y \in S,     \lb{E.0b} \\
& \rho(x,y)\leq\rho(x,z)+\rho(z,y), \quad x,y,z\in S.     \lb{E.0c}
\end{align}
We point out that a semi-metric space may have two distinct elements $x,y\in S$ with $\rho(x,y)=0$. Nevertheless, one can introduce the notion of Cauchy sequences and limits in a semi-metric space
$(S,\rho)$ as usual. Of course, in this case convergent sequences may have several distinct limits. A semi-metric space $(S,\rho)$ is called {\it complete} if every Cauchy sequence of points in $S$ has a limit which also lies in $S$. Next, we discuss two approaches to completion of a semi-metric space.

\noindent

$\mathbf{(I)}$ Given a semi-metric space $(S,\rho)$, one introduces a semi-metric space
$\big(\wti S_1,\wti\rho_1\big)$, where $\wti S_1$ is the set of all Cauchy sequence in $S$ and
\begin{align}
\wti \rho_1(\wti x,\wti y)=\lim_{n\to\infty}\rho(x(n),y(n)), \quad
\wti x=\{x(n)\}_{n\in\bbN}, \wti y=\{y(n)\}_{n\in\bbN}\in\wti S_1.     \lb{E.2}
\end{align}
It follows from the triangle  inequality \eqref{E.0c} for $\rho$ that
\begin{equation}
|\rho(x(n),y(n))-\rho(x(m),y(m))|\leq\rho(x(n),x(m))+\rho(y(n),y(m)),
\end{equation}
and hence for all $\wti x,\wti y\in\wti S_1$ the sequence $\{\rho(x(n),y(n))\}_{n\in\bbN}$ is Cauchy in $\bbR$. Thus, the limit in \eqref{E.2} exists, and using \eqref{E.2} one verifies that $\wti \rho_1$ is a semi-metric on $\wti S_1$. Moreover, it has been shown in 
\cite[p.\ 176]{Wi83} that $\big(\wti S_1,\wti \rho_1\big)$ is a complete semi-metric space and that $(S,\rho)$ is isometric to a dense subset in $\big(\wti S_1,\wti \rho_1\big)$. Introducing an equivalence relation $\sim$ on $\wti S_1$ by
\begin{equation}
\wti x\sim \wti y \, \text{ whenever } \, \wti \rho_1(\wti x,\wti y)=0, \quad \wti x,  \wti y \in \wti S_1,
\end{equation}
one defines the set of equivalence classes
\begin{equation}
S_1 = \wti S_1/{\sim} = \{[\wti x]_{\wti p_1} \,|\, \wti x\in\wti S_1\}
\end{equation}
and a metric on $S_1$ by
\begin{equation}
\rho_1([\wti x]_{\wti \rho_1},[\wti y]_{\wti \rho_1})=\wti \rho_1(\wti x,\wti y), \quad
[\wti x]_{\wti \rho_1}, [\wti y]_{\wti \rho_1} \in S_1.
\end{equation}
It follows from the triangle inequality for $\wti \rho_1$ that
\begin{equation}
\wti \rho_1(\wti x_1,\wti y_1)=\wti \rho_1(\wti x_2,\wti y_2) \, \text{ whenever } \,
\wti x_1\sim \wti x_2 \, \text{ and } \, \wti y_1\sim \wti y_2,
\end{equation}
and hence $\rho_1$ is a well-defined metric on $S_1$. Thus, $(S_1,\rho_1)$
is a complete metric space, and $(S,\rho)$ is isometric to a dense subset of $(S_1,\rho_1)$.

Now, we consider a different approach:

\noindent
$\mathbf{(II)}$ Given a semi-metric space $(S,\rho)$, define an equivalence relation $\sim$ on $S$ by
\begin{equation}
x\sim y \, \text{ whenever } \, \rho(x,y)=0, \quad x, y \in S,
\end{equation}
and the set of equivalence classes
\begin{equation}
M = S/{\sim} = \{[x]_{\rho} \,|\, x\in S\}.
\end{equation}
It follows from the triangle  inequality for $\rho$ that for all 
$x_1,x_2,y_1,y_2\in S$
\begin{equation}
\rho(x_1,y_1)=\rho(x_2,y_2) \, \text{ whenever } \, 
x_1\sim x_2 \, \text{ and } \, y_1\sim y_2.
\end{equation} 
Hence, the function
\begin{equation}
d([x]_{\rho},[y]_{\rho})=\rho(x,y), \quad [x]_{\rho}, [y]_{\rho} \in M,
\end{equation}
is a well-defined metric on $M$. One completes the metric space $(M,d)$ by introducing the set
$\hatt S_2$ of all Cauchy sequences of points in $M$ and a semi-metric
\begin{align}
\hatt \rho_2(\hatt x,\hatt y) = \lim_{n\to\infty}d([x]_\rho(n),[y]_\rho(n)),
\quad  \hatt x=\{[x]_\rho(n)\}_{n\in\bbN},  \hatt y=\{[y]_\rho(n)\}_{n\in\bbN}\in\hatt S_2.     \lb{E.1}
\end{align}
Thus, $\big(\hatt S_2,\hatt \rho_2\big)$ is a complete semi-metric space. Introducing an equivalence relation $\sim$ on $\hatt S_2$ by
\begin{equation}
\hatt x\sim \hatt y \, \text{ whenever } \, \hatt \rho_2(\hatt x,\hatt y)=0, \quad
\hatt x, \hatt y \in \hatt S_2,
\end{equation}
one defines the set of equivalence classes
\begin{equation}
S_2 = \hatt S_2/{\sim} = \{[\hatt x]_{\hatt \rho_2} \,|\, \hatt x\in\hatt S_2\}
\end{equation}
and a metric on $S_2$ by
\begin{equation}
\rho_2([\hatt x]_{\hatt \rho_2},[\hatt y]_{\hatt \rho_2})=\hatt \rho_2(\hatt x,\hatt y), \quad
[\hatt x]_{\hatt \rho_2}, [\hatt y]_{\hatt \rho_2} \in S_2.
\end{equation}
Again, it follows from the triangle  inequality for $\hatt \rho_2$ that 
$\rho_2$ is a well-defined metric on $S_2$. Thus, $(S_2,\rho_2)$ is a complete metric space, and $(M,d)$ and hence $(S,\rho)$ are isometric to a dense subset of $(S_2,\rho_2)$.

The main result of this appendix is the following isometry lemma.

\begin{lemma} \lb{lE.1}
The metric spaces $(S_j,\rho_j)$, $j=1,2$, introduced in steps $(I)$ and $(II)$ above are isometric
$($i.e., there exists an isometric bijection $J:S_1\to S_2$$)$.
\end{lemma}
\begin{proof}
To establish an isometric bijection $T:S_1\to S_2$ one proceeds as follows: For every element
$\wti x=\{x(n)\}_{n\in\bbN}\in\wti S_1$ one defines $\hatt x\in\hatt S_2$ by
$\hatt x=\{[x(n)]_\rho\}_{n\in\bbN}$. Then
\begin{align}
\begin{split}
& \rho_1([\wti x]_{\wti \rho_1},[\wti y]_{\wti \rho_1}) = \wti \rho_1(\wti x,\wti y)
= \lim_{n\to\infty}\rho(x(n),y(n))
\\
& \quad = \lim_{n\to\infty}d([x]_\rho(n),[y]_\rho(n)) = \hatt \rho_2(\hatt x,\hatt y)
= \rho_2([\hatt x]_{\hatt \rho_2},[\hatt y]_{\hatt \rho_2}).       \lb{E.3}
\end{split}
\end{align}
Since $\rho_2$ is a metric, it follows from \eqref{E.3} that for every $\wti x,\wti y\in \wti S_1$ with $[\wti x]_{\wti \rho_1}=[\wti y]_{\wti \rho_1}$ the above construction yields $[\hatt x]_{\hatt \rho_2}=[\hatt y]_{\hatt \rho_2}$. Thus, there is no ambiguity in defining an isometry
\begin{equation}
J: \begin{cases} S_1\to S_2 \\
[\wti x]_{\wti \rho_1} \mapsto J([\wti x]_{\wti \rho_1}) = [\hatt x]_{\hatt \rho_2}
\end{cases}
\end{equation}
(i.e., $J$ is well-defined). It follows from the above constructions that the domain of $J$ is all of
$S_1$ and the range is all of $S_2$. Moreover, since $\rho_1$ is a metric, \eqref{E.3} implies
that $J$ is one-to-one and hence a bijection.
\end{proof}

\begin{remark}
In the case of $(S,\rho)$ being a semi-normed vector space, that is,
\begin{equation}
\rho(x,y)=\|x-y\|, \quad x, y \in S,
\end{equation}
the spaces $(S_1,\rho_1)$ and $(S_2,\rho_2)$ become Banach spaces, and  in this case
\begin{equation}
S_1=\wti S_1 / \ker(\wti \rho_1), \quad M=S/\ker(\rho), \quad S_2=\hatt S_2 / \ker(\hatt \rho_2),
\end{equation}
where $\ker(\wti \rho_1)$, $\ker(\hatt \rho_2)$, and $\ker(\rho)$ denote the linear subspaces of elements of zero norm,
\begin{equation}
\ker(\rho)=\{x\in S \,|\, \|x\|=0\},
\end{equation}
and similarly for $\ker(\wti \rho_1)$ and $\ker(\hatt \rho_2)$. In addition, the analog of the isometry
$J$ in Lemma \ref{lE.1} now also becomes a linear map in the context of normed spaces.
\end{remark}

\medskip

{\bf Acknowledgments.} 
We are indebted to Nigel Kalton, Konstantin Makarov, Mark Malamud, and Eduard Tsekanovskii for numerous helpful discussions on various aspects of this subject. 

This paper is dedicated to Jerry Goldstein on the occasion of his 70th birthday. Jerry 
has been a mentor and friend to us for many years; his enthusiasm for mathematics has 
been infectious and will always remain an inspiration to us.


\end{document}